\documentclass[11pt]{amsart}
\usepackage{graphicx,color}
\usepackage{latexsym}
\usepackage{graphicx} 
\usepackage{amsthm,amsmath, amssymb,amsfonts,esint}
\usepackage[bookmarks=true]{hyperref}

\usepackage{amsmath,amssymb,amsfonts,amsthm}
\usepackage{layout,textcomp}

\theoremstyle{plain}

\theoremstyle{plain}
\newtheorem{theorem}{Theorem} [section]

\newtheorem{lemma}[theorem]{Lemma}
\newtheorem{proposition}[theorem]{Proposition}

\theoremstyle{definition}

\theoremstyle{remark}
\newtheorem{remark}[theorem]{Remark}

\numberwithin{theorem}{section}
\numberwithin{equation}{section}
\numberwithin{figure}{section}

\newcommand{\pl}{\left(}
\newcommand{\pr}{\right)}
\newcommand{\Ric}{\text{\rm Ric}}  
\newcommand{\Sph}{\mathbb{S}}

\newcommand{\He}{\text{\rm Hess}}

\begin{document}

\title[Metrics with non-negative Ricci on convex three-balls]{Metrics with non-negative Ricci curvature on convex three-manifolds}
\author[Antonio Ach\'{e}]{Antonio Ach\'{e}}
\address{Department of Mathematics, Princeton University, Fine Hall, Washington Road, Princeton, NJ 08544}
\email{aache@math.princeton.edu}

\author[Davi Maximo]{Davi M{a}ximo}
\address{Department of Mathematics, Stanford University, 450 Serra Mall, Bldg 380, Stanford, CA 94305}
\email{maximo@math.stanford.edu}

\author[Haotian Wu]{Haotian Wu}
\address{Department of Mathematics, University of Oregon, Eugene, OR 97403}
\email{hwu@uoregon.edu}

\thanks{AA was partially supported by a postdoctoral fellowship of the
National Science Foundation, award No. DMS-1204742.}

\date{}
\begin{abstract}
We prove that the space of smooth Riemannian metrics on the three-ball with non-negative Ricci curvature and strictly convex boundary is path-connected; and, moreover, that the associated moduli space ($i.e.$, modulo orientation-preserving diffeomorphisms of the three-ball) is contractible. As an application, using results of Maximo, Nunes, and Smith \cite{MaNuSm13}, we show the existence of properly embedded free boundary minimal annulus on any three-ball with non-negative Ricci curvature and strictly convex boundary.

\end{abstract}
\maketitle

\section{Introduction}


Let $M$ be a compact three-manifold with non-empty boundary $\partial M =D$. We consider smooth Riemannian metrics on $M$ that have non-negative Ricci curvature and strictly convex\footnote{We recall that a Riemannian manifold $(M,g)$ is said to have strictly convex boundary if every boundary point has strictly negative second fundamental form with respect to the outward-pointing unit normal. This infinitesimal definition of convexity is equivalent to other more geometric conditions, see Bishop \cite{Bis74}.} non-empty boundary $\partial M$. By a variational argument developed in Meeks, Simon, and Yau \cite{MeSiYa82}, such metrics can exist if and only if $M$ is diffeomorphic to the three-ball in $\mathbb{R} ^{3}$ (cf. Fraser and Li \cite{FrMa14}). In this note, we are interested in studying the space of such metrics. We prove:

\begin{theorem}\label{thmmain}
The space of smooth Riemannian metrics on the three-ball ${M}^3$ with non-negative Ricci curvature and strictly convex boundary is path-connected. Moreover, the associated moduli space ($i.e.$, modulo orientation-preserving diffeomorphisms of the three-ball) is contractible. 
\end{theorem}

As an intermediate step in the proof of Theorem \ref{thmmain}, we also show that:

\begin{theorem}\label{thmrc+}
The space of smooth Riemannian metrics on the three-ball ${M}^3$ with positive Ricci curvature and strictly convex boundary is path-connected. Moreover, the associated moduli space ($i.e.$, modulo orientation-preserving diffeomorphisms of the three-ball) is contractible.
 \end{theorem}

The study of the topology of the space of metrics satisfying certain curvature conditions has a long history. In 1916, Weyl \cite{Weyl16} showed that the space of metrics with positive scalar curvature on the two-sphere $S^2$ is path-connected and while his proof is clearly two-dimensional (as it uses the Uniformization Theorem for surfaces), it is natural to ask whether or not analogues to Weyl's result could hold in higher dimensions. Using Ricci flow on closed three-manifolds, Hamilton \cite{Ham82} showed that the space of metrics with positive Ricci curvature is path-connected. More recently, Marques \cite{Mar12}, using Ricci flow with surgeries, proved the path-connectedness of the space of metrics with positive scalar curvature on three-manifolds. 

The picture in higher dimensions is quite different. As first observed in the work of Hitchin \cite{Hit74}, the spaces of metrics of positive scalar curvature on the spheres $S�^{8k}$ and $S�^{8k+1}$, respectively, are disconnected for each $k\geq1$. Some years later, Carr \cite{Car88} proved that the space of metrics with positive scalar curvature on $S�^{4k-1}$ has infinitely many connected components for each $k\geq2$, see also \cite{GrLa83}. This was strengthened by Kreck and Stolz \cite{KrSt93}, who proved that the space of such metrics has infinitely many connected components for $k\geq 2$ even modulo diffeomorphisms ($i.e.$, the moduli space), see also \cite{BoGi96}. Similar results were also obtained for Ricci curvature by Wraith \cite{Wra11}, who proved that there are closed manifolds in infinitely many dimensions for which the moduli space of metrics with positive Ricci curvature has infinitely many components.


Finer aspects of the topology of the space of metrics, such as the fundamental group or higher homotopy groups, have also been studied. For instance, Rosenberg and Stolz \cite{RS01} have shown that the space of metrics with positive scalar curvature on $S^2$ is actually contractible; and it is not hard to prove, using Hamilton's Ricci flow \cite{Ham82}, that so is the moduli space of metrics with positive Ricci curvature on closed three-manifolds. In higher dimensions, the higher homotopy groups for the space of metrics with positive scalar curvature have been studied for certain closed manifolds, see for example \cite{BHSW10,HaScSt14}, and also recently for manifolds with boundary by Walsh \cite{Wal14}.

The question we answer in Theorem \ref{thmmain} came to our attention after the recent solution by Maximo, Nunes, and Smith \cite{MaNuSm13} to a problem of Jost \cite{Jos88}: every strictly convex domain of $\mathbb{R}�^{3}$ contains a properly embedded free boundary minimal annulus. Here, a smooth compact surface $\Sigma$ in $(M,g)$ with $\partial\Sigma\subseteq\partial M$ is said to be {minimal} with free boundary whenever it has zero mean curvature and $T\Sigma$ is orthogonal to $T\partial M$ at every point of $\partial\Sigma$. The argument used in \cite{MaNuSm13} is based on degree theoretic considerations and works for any metric with non-negative Ricci curvature and suitably convex boundary. As an application of Theorem \ref{thmmain}, it is not hard to extend their result to any metric with non-negative Ricci curvature and strictly convex boundary in the usual sense (assuming the main results of \cite{MaNuSm13}, we provide the details of Theorem \ref{thmmns} in Section \ref{sec:app}): 

\begin{theorem}\label{thmmns}
Let $g$ be a Riemannian metric on the three-ball $M^3$ with non-negative Ricci curvature and strictly convex boundary. There exists a properly embedded free boundary minimal annulus in $M^3$.
\end{theorem}

We finish the introduction by saying a few words about the proof of Theorem \ref{thmmain} and the ideas involved. Let us first explain the proof for the path-connectedness, which is the first as well as the essential step in our proof. One could imagine a possible approach to this problem by using Ricci flow with boundary, or by flowing the boundary inward through mean curvature flow, or even by running Ricci flow coupled with mean curvature flow. While there have been interesting recent results under these settings, for instance \cite{Hui85, Gia12, Lot12, Bre13}, they do not seem to work for the question at hand. Indeed, the long-time behavior of the flows proposed in \cite{Gia12} and \cite{Lot12} is not well understood, and also there is no reason why the convexity of the boundary should be preserved while flowing it inside by mean curvature in \cite{Bre13}. Therefore, we pursue a different strategy. Our first step, which might be of independent interest, is to use an idea of Perelman to deform the metric near the boundary making it totally geodesic while maintaining (not necessarily strict) convexity and non-negative Ricci curvature. We then glue two copies of the deformed manifolds along the boundary, obtaining a closed manifold diffeomorphic to the three-sphere $\Sph�^{3}$ with a metric of non-negative Ricci curvature and a reflection symmetry across the surface along which the gluing occurs. The idea is then to run Ricci flow on this glued manifold, and to argue by Hamilton's result \cite{Ham82} that it will flow to a round sphere (after normalization). Since the reflection symmetry is preserved along Ricci flow, we can then find a path of metrics with non-negative Ricci curvature and convex boundary (not necessarily strictly so, since the boundary will be totally geodesic along the flow) from the original metric to the standard round metric of a hemisphere of $\Sph^3$, and we show that such path can always be deformed into a path of metrics with non-negative Ricci curvature and strictly convex boundary. This whole procedure works modulo actions of the group of orientation-preserving diffeomorphisms on the three-ball, which is path-connected by Cerf \cite{Cerf}.

Once the path-connectedness is established, we deal with the higher homotopy groups. To this end, we show that, for any fixed compact set of initial metrics, the above procedure can be done in such a way that the paths obtained depend continuously on the initial data.  

The paper is organized as follows. Section \ref{sect2} is devoted to the proof of Theorem \ref{thmrc+}. We then prove Theorem \ref{thmmain} in Section \ref{sect3}. In Section \ref{sec:app}, we explain the proof of Theorem \ref{thmmns}. We include some computations in the Appendix for completeness.

\subsection*{Acknowledgments} We are grateful to S\o ren Galatius, James Isenberg, Richard Schoen, and Gang Tian for their interest in and helpful discussions on this project. We also thank the editors and the referee for valuable comments on our manuscript.

\section{Proof of Theorem \ref{thmrc+}}\label{sect2}
Let $\mathcal{C}$ be the space of smooth Riemannian metrics on the three-ball ${M}^3$ with positive Ricci curvature and strictly convex boundary. Let $\mathcal{C}_0$ be the space of smooth Riemannian metrics on the three-ball ${M}^3$ with positive Ricci curvature and convex boundary. Note that since we do not necessarily assume the boundary convexity to be strict for metrics in $\mathcal{C}_{0}$, $\mathcal{C}$ is a strict subset of  $\mathcal{C}_{0}$.

We will use a certain type of metric deformation which we call {\it shifts} that can be defined as follows. Given a metric $g$ on $M$, we consider the equidistant surfaces to the boundary $D=\partial M$. Suppose that $\Sigma$ is one of such surfaces and that the distance with respect to $g$ from $\Sigma$ to the boundary is $\varepsilon$. If $\varepsilon$ is sufficiently small, $\Sigma$ must be smooth and isotopic to $D$ in $M$. Moreover, the set of points whose distance to $D$ is greater than or equal to $\varepsilon$ defines a three-manifold $N$ with boundary $\partial N=\Sigma $. It is clear that $(M,\partial M)$ is diffeomorphic to $(N,\partial N)$ as both manifolds are diffeomorphic to $(B^3, S^2)$. An {\it $\varepsilon$-shift} of $g$ is then the metric obtained by restricting $g$ to $(N,\partial N)$ and then pulling it back to $(M,\partial M)$ by the above diffeomorphism. Note that if we can do a shift for a distance $\varepsilon$, then we can also do shifts for any distance smaller than $\varepsilon$, and moreover, we obtain a path between $g$ and the shifted metrics. Furthermore, if $g\in\mathcal{C}$ and the shift parameter is sufficiently small, with smallness depending on $g$, all the shifted metrics will also belong to $\mathcal{C}$. 

We will also deform metrics by Ricci flow $-$ after appropriate doubling of $M$. Note that both Ricci flow and shifts are well-defined up to actions of the group $\text{Diff}_{+} := \text{Diff}_+(B^3)$ of orientation-preserving diffeomorphisms of the three-ball. Our goal is to show that $\mathcal C$ is path-connected and $\mathcal{C}/ \text{Diff}_+$ is contractible. 

\subsection{$\mathcal{C}$ is path-connected}\label{secCpath} Since $\text{Diff}_+$ is path-connected \cite{Cerf}, we may ignore its action over $\mathcal{C}$ through out the remainder of Section \ref{secCpath}. We will show that any metric $g$ in $\mathcal{C}$ can be connected through a path in $\mathcal{C}$ to a metric on some strictly convex round spherical cap. This is carried out in three steps. Given any $g \in \mathcal{C}$, we prove:
\begin{itemize}
\item[(A)] There exists a path in $\mathcal{C}_0$ connecting $g$ to a metric $g_2$ with positive Ricci curvature and totally geodesic boundary. 
\item[(B)] There exists a path in $\mathcal{C}_0$ between $g_{2}$ and the metric of the standard round metric on a hemisphere.
\item[(C)] One can deform the above paths into a path in $\mathcal{C}$ from $g$ to the metric of a strictly convex round spherical cap.
\end{itemize}

\subsubsection{\bf A: connecting $g$ to a metric with totally geodesic boundary}\label{stepA}
\begin{proposition}\label{propmain}
Given a metric $g \in \mathcal{C}$, there exists a path in $\mathcal{C}_0$ connecting $g$ to a metric $g_2$ with positive Ricci curvature and totally geodesic boundary.
\end{proposition}

We prove Proposition \ref{propmain} by first constructing an appropriate metric  $g_2$ with positive Ricci curvature and totally geodesic boundary and then showing how to connect it to $g$ by a path in $\mathcal{C}_{0}$.  

The existence of $g_{2}$ is based on the following gluing construction. Suppose we were to glue two copies of $(M,g)$ along the boundary $\partial M$ so that, after smoothing the edges inside a neighborhood $U$ of the boundary, we obtain a smooth metric on the glued manifold (which is topologically a three-sphere) with positive Ricci curvature and a reflection symmetry across $\partial M$. Then, half of the glued manifold will be carrying a metric on the ball $B�^{3}$ with totally geodesic boundary (the latter because of the reflection symmetry), which is the metric $g_{2}$ we were looking for. Moreover, this gluing is done such that we can connect $g$ to $g_{2}$ by a path in $\mathcal{C}_{0}$. 

\begin{remark}
Our gluing method uses ideas from Perelman \cite{Per97}, where he shows that one can always glue two compact manifolds with positive Ricci curvature with isometric boundary while keeping the Ricci curvature positive, provided that the second fundamental forms of the boundary points in one of the manifolds are {\it strictly} bigger than the negatives of the mean curvatures of the corresponding boundary points of the other manifold. What we show is that one can carry out a similar construction when the two manifolds to be glued are strictly convex and identical and also that, if one is careful enough, the metrics involved can actually be connected by a path in $\mathcal{C}_{0}$.

\end{remark}

\noindent {\it Gluing and construction of $g_{2}$.} 
To construct $g_{2}$, we glue two copies of $M$ along the boundary $D$ to obtain a manifold $M\#_{D}M$ that is diffeomorphic to the three-sphere. The doubled metric on $M\#_{D}M$ is smooth away from the two-sphere $D$. We will abuse notation slightly and denote that metric on $M\#_{D}M$ also by $g$. On $M\#_{D}M$, consider the parametrization of a tubular neighborhood of $D$ given by equidistant surfaces that are two-spheres. Working in such neighborhood, for an $\varepsilon$ small enough, we can find a local expression for $g$ given by
$$g=dr^2+g^r,$$ 
where $r\in (-\varepsilon, \varepsilon)$ and $g�^{r}$ is a metric on the two-sphere $S�^{2}$. Here, $r$ is the signed-distance to $D$ in $M\#_{D}M$ with respect to $g$, which we can define once we choose an orientation for $D$. Since our final goal is to obtain a metric in $M$ and not in $M\#_{D}M$, it will be important for us to fix the orientation once and for all: we will always think of the half containing $\{r \leq 0\}\times S^2$ as our original manifold $M$. In particular, if we choose $\varepsilon>0$ small enough, since $(M,g)$ has strictly convex boundary, we can see that the equidistant surfaces $\Sigma_{r}$ (of signed-distance $r$ to $D$) obtained by fixing $r$ negative (respectively, $r$ positive) are smooth two-spheres isotopic to the boundary $D$ and strictly convex (respectively, concave) with respect to $g$.  
 
To start the gluing argument, we choose a parameter $\rho\in (0, \varepsilon)$ and consider the surfaces $\Sigma_{-\rho}$ and $\Sigma_{\rho}$. The idea is to interpolate the values of $g$ in $\Sigma_{-\rho}$ and $\Sigma_{\rho}$ to construct a new metric $\bar{g}$ which has positive Ricci curvature and a reflection symmetry across $D$. The new metric $\bar{g}$ on $[-\rho,\rho] \times D $ will have the form
 \begin{align*}
 \bar{g}= dr^2 + \bar{g}^r_\rho(x), 
 \end{align*} 
where $\bar{g}^r_\rho$ is a quadratic polynomial in $r$ interpolating, up to first order, between $g^{-\rho}$ and $g^\rho$. More precisely, let $(r,x)$ be local coordinates representing a tubular neighborhood near $D$. Then in these local coordinates, we have
$$\left(\bar{g}^r_\rho(x)\right)_{ij}=b_{ij}(x)r^{2}+c_{ij}(x), $$
where $b_{ij}$ and $c_{ij}$ are determined by imposing the condition that $\bar{g}$ matches the metric on $[-\rho,\rho]\times D$ to first order at the boundary. This only means that at the slices $\Sigma_{-\rho}$ and $\Sigma_{\rho}$, we must have (indices $i,j$ representing directions tangential to $D$)
 \begin{align}
\left. \left(\bar{g}^{r}_{\rho}(x)\right)_{ij} \right\vert_{r=-\rho}&=b_{ij}(x)\rho^{2}+c_{ij}(x)\label{geq1},\\
\left. {\partial_r} \left(\bar{g}^r_{\rho}(x)\right)_{ij} \right\vert_{r=-\rho}&=-2b_{ij}(x)\rho\label{geq2}.
 \end{align} 
From \eqref{geq1} and \eqref{geq2} we have that when $r=-\rho$,
 \begin{align}
 b_{ij}(x)&=-\frac{\left(g^{-\rho}\right)^{'}_{ij}(x)}{2\rho}\label{coeff1},\\
 c_{ij}(x)&=g^{-\rho}_{ij}(x)+\frac{\left(g^{-\rho}\right)^{'}_{ij}(x)}{2}\rho,\label{coeff2} 
 \end{align}  
where the prime derivatives are taken with respect to $\partial_r$. Observe that we have only prescribed a matching condition at the slice $\Sigma_{-\rho}$, because by symmetry, the matching condition at the slice $\Sigma_{\rho}$ will yield exactly the same coefficients. By convexity of $\Sigma_{-\rho}$, $(g^{-\rho})^{'}_{ij}(x)>0$, so $b_{ij}(x)<0$ for every $x\in D$. Moreover, using compactness of $D$, we can choose $\Lambda>0$ such that the eigenvalues of $(g^{-\rho})'(x)$ are bigger than $2\Lambda$ for every $x\in D$. This implies that
\begin{align}\label{eqper}
(\bar{g}^r_\rho)''<-\frac{\Lambda}{\rho}\bar{g}^r_\rho \quad \text{on}\,\, (-\rho,\rho)\times D.
\end{align}
Note that $\bar{g}$ as constructed above is a smooth metric except at the hypersurfaces $\Sigma_{-\rho}$ and $\Sigma_\rho$, where it is only ${C}^1$.


Regarding the curvature of $\bar{g}$, we show:

\begin{lemma}\label{lemmarc}
In the above, the parameter $\rho$ can be chosen small enough such that $\bar{g}$ has positive Ricci curvature wherever it is smooth.
\end{lemma}
The proof of Lemma \ref{lemmarc} follows by a direct calculation and depends crucially on equation \eqref{eqper}. We refer the reader to the Appendix for details. In particular, note that equation \eqref{eqper} holds true in general dimension, not just in dimension three, so Lemma \ref{lemmarc} is true in any dimension two or higher.

In the remainder of this subsection, we consider the metric $\bar{g}$ restricted to the original manifold $M$. We note that if $d_{g}$ and $d_{\bar{g}}$ denote respectively the distance function with respect to $g$ and $\bar{g}$, then, since $g=dr�^{2}+g�^{r}$ and $\bar{g}=dr�^{2}+\bar{g}�^{r}$, we have for $r\in [-\varepsilon,0]$:
\begin{equation}\label{eq:sur}
\{x\in M\,|\, d_{\bar{g}}(x,\partial M) =-r\}=\{x\in M\,|\, d_{g}(x,\partial M)=-r \} =: \Sigma_{r}.
\end{equation}
It is also of interest to compute the second fundamental form of $\Sigma_{r}$ with respect to $g$ and $\bar{g}$ respectively. Direct calculation yields that\footnote{Here, the second fundamental form $\text{\rm II}$ is computed with respect to the outward-pointing normal, so $\text{\rm II}^r = -(g^r)'$; convexity of $\Sigma_r$, $r\in(-\varepsilon, 0)$, means that $(g^r)'>0$.}
\begin{align}
-\text{\rm II}�^{r}_{g} &= (g�^{r})'> 0,\,r\in [-\varepsilon,0];  \nonumber\\ 
-\text{\rm II}�^{r}_{\bar{g}} &= (\bar{g}�^{r})'>0,\, r\in [-\varepsilon,0); \quad \text{\rm II}�^{r}_{\bar{g}}\equiv 0,\, r=0. \label{eqremsur1}
\end{align}

Finally, to obtain the metric $g_2$, we smooth out $\bar{g}$ near $\Sigma_{- \rho}$ as follows\footnote{The smoothing construction follows an argument sketched by Perelman \cite{Per97}. Details can be found in \cite{Wan97}, page 26.}. For $r=-\rho$, note that $\bar{g}$ is $C^1$ and has one-sided second derivatives. Now, the Ricci curvature has the form
\begin{align*}
\Ric(\bar{g})=\bar{g}^{-1}\partial^2\bar{g} + Q(\partial \bar{g}, \bar{g}),
\end{align*}
which is linear in the second derivatives of $\bar{g}$. Since $\bar{g}$ is $C�^{1}$, we can interpolate between $\bar{g}$ and a smooth metric near the surface $\Sigma_{-\rho}$ and make it in such a way that the smooth metric will also have positive Ricci curvature. Moreover, since $\bar{g}=dr�^{2}+\bar{g}�^{r}$, the interpolation actually occurs along the factor $\bar{g}�^{r}$, for $r$ near $-\rho$. Being so, we can directly see that in the same local coordinates $g_{2} = dr^2 + \bar{g}_2^r$ will satisfy similar conditions as in \eqref{eq:sur} and \eqref{eqremsur1}, that is, 
\begin{align}
\Sigma_{r} & =\{x\in M\,|\, d_{g_{2}}(x,\partial M)=-r \}, r\in  [-\varepsilon,0];\label{eqg21} \\
-\text{\rm II}�^{r}_{{g_{2}}} & = (\bar{g}�^{r}_{2})'>0,\, r\in  [-\varepsilon,0);  \quad \text{\rm II}�^{r}_{{g_{2}}}\equiv 0,\, r=0. \label{eqg22}
\end{align}

\noindent{\it Connecting $g$ to $g_{2}$ by a path in $\mathcal{C}_{0}$}. Starting with $g \in\mathcal{C}$, for a choice of $\varepsilon$ small as above, the $\varepsilon$-shift of $g$, which we call $g_{1}$, belongs to $\mathcal{C}$ and is path-connected to $g$. Moreover, since $\rho<\varepsilon$, $g_{1}$ is also a shift of $g_{2}$, and the path connecting them belongs to $\mathcal{C}$, with the exception of $g_{2}$, which belongs to $\mathcal{C}_{0}$. This completes the proof of Proposition \ref{propmain}.

\subsubsection{\bf B: connecting $g_2$ to a metric of a round hemisphere by a path in $\mathcal{C}_{0}$} We next prove that it is possible to connect the metric $g_{2}$ to the round metric of a hemisphere of $\Sph�^{3}$. We accomplish this by using Ricci flow. By construction, and abusing notation slightly, we can consider the metric $g_{2}$ as a smooth metric with positive Ricci curvature on the doubled manifold $M\#_{D}M$ and with a reflection symmetry across $D$. We then run Ricci flow
$$\partial_{t}g=-2\Ric({g})$$
with $g(0)=g_{2}$. By work of Hamilton \cite{Ham82}, if we denote by $g(t)$, $t \in [0, T)$, the
unique maximal solution to the above Ricci flow, then the rescaled metrics $\hat{g}(t)=\frac{1}{4(T-t)}g(t)$ converge to a metric of constant curvature 1 as $t\nearrow T$. Since Ricci flow preserves isometries, in particular, it will preserve the reflection symmetry of $g_{2}$ across $D$. Therefore, $D$ remains a totally geodesic two-sphere in $M\#_{D}M$ with respect to $\frac{1}{4(T-t)}g(t)$. Restricting the metrics $\hat{g}(t)$ to $M$, we obtain a path in $\mathcal{C}_{0}$ from $g_{2}$ to the round metric $g_{h}$ of a hemisphere of $\Sph�^{3}$. 

\subsubsection{\bf C: connecting $g$ to a metric of a strictly convex round spherical cap by a path in $\mathcal{C}$}\label{stepC} Along the above path obtained by Ricci flow, the metrics $\hat{g}(t)$ are uniformly equivalent, \emph{i.e.}, there exists a constant $C>0$ such that $C^{-1}\hat{g}(0) \leq \hat{g}(t)\leq C \hat{g}(0)$ \cite{Ham82}. Therefore, arguing by compactness, we can find space-time neighborhood of $\partial M$ of the form $[-r_{0},r_0]\times S^2 \times [0,T]$, such that for each $t\in[0,T]$, the metric $\hat{g}(t)$ can be written as
\begin{align*}
\hat{g}(t)=dr^2+{\hat{g}}^{r}_{t}.
\end{align*}
Let $\varphi(r)$ be a smooth function of one real variable and of sufficiently small $C�^{2}$-norm,  supported in the interval $(-r_{0},r_{0})$, and such that $\varphi{'} (0)>0$. We perturb the metric $\hat{g}(t)$ by adding a term of the form $\eta\varphi(r){\hat{g}}^{r}_{t}$ to get
\begin{align}\label{defgtilde}
\tilde{g}(t)=\hat{g}(t)+\eta\varphi(r){\hat{g}}^{r}_{t},
\end{align}
where $\eta>0$ is a small number to be chosen. Note that this metric is well-defined in all of $M\#_{D}M\times[0,T]$ because $\varphi$ is supported in $(-r_{0},r_{0})$ and at the points whose $\hat{g}(t)$-distance to the boundary is less than $r_{0}$ we can write
\begin{align*}
\tilde{g}(t)=dr^2+\left(1+\eta\varphi(r)\right){\hat{g}}^{r}_{t}.
\end{align*}
This perturbation does not change the unit normal to the boundary. Because of this, the second fundamental form at the boundary $D$  with respect to $\tilde{g}(t)$ is given by 
\begin{align*}
-\text{\rm II}_{\tilde g (t)}&=\left.\frac{\partial}{\partial r}\right\vert_{r=0} \left(1+\eta\varphi(r)\right){\hat{g}}^{r}_{t} \\
&= \left(1+\eta\varphi(0)\right)\left.\left({\hat{g}}^{r}_{t}\right)'\right\vert_{r=0} + \left. \eta\varphi{'} (0)\hat{g}^r_{t}\right\vert_{r=0}\\
&=\eta\varphi{'}(0){\hat{g}}^0_{t}, 
\end{align*}
which says that $(M,\tilde{g}(t))$ has  strictly convex boundary $D$, since $\varphi{'}(0)>0$ and $\eta>0$ (we have used that the boundary is totally geodesic under $\hat{g}(t)$). Since positive Ricci is an open condition, we can choose $\eta>0$ small enough such that $\tilde{g}(t)$ has positive Ricci curvature for all $t\in[0,T]$, and this procedure deforms the path connecting $g_2$ and $g_h$ in $\mathcal{C}_0$ into a path connecting $\tilde{g}_2$ and $\tilde{g}_h$ in $\mathcal{C}$.

At this point,  given any $g\in\mathcal{C}$, we have constructed the following paths of metrics:
\begin{align}\label{eq:pathhemi}
g \xrightarrow{\alpha} g_1 \xrightarrow{\beta} g_2 \xrightarrow{\gamma} g(T) = g_h,
\end{align}
where $\alpha \in\mathcal{C}$; all the metrics along the path $\beta$ belong to $\mathcal{C}$, with the exception of the endpoint $g_{2}$; and $\gamma$ is a path in $\mathcal{C}_{0}$ given by Ricci flow, which can be deformed to a path $\tilde{\gamma}$ in $\mathcal{C}$ between the metrics $\tilde{g}_{2} = g_{2} + \eta\varphi(r)g^{r}_{2}$ and $\tilde{g}_{h}={g}_{h}+\eta\varphi(r){{g}}^{r}_{h}$ by the above procedure. 

To finish proving the claim that $\mathcal{C}$ is path-connected, we just need to construct paths in $\mathcal{C}$ connecting $\tilde g_2$ to a suitable shift of $g_2$, and $\tilde g_h$ to a suitable shift of $g_h$, respectively.

We note that, for $s\in[0,1]$,  the metrics
\begin{align}
\theta(s)=g_{2} + (1-s)\eta\varphi(r)g^{r}_{2}
\end{align}
induce a path between $\tilde{g}_{2}$ and $g_{2}$ that belongs to $\mathcal{C}$ for $s\in[0,1)$. Composing this path with a continuous families of shifts near $s=1$ will yield a path in $\mathcal{C}$ between $\tilde{g}_{2}$ and a small shift of $g_{2}$ which lies somewhere in path $\beta$. More precisely, even though the boundary of $M$ is totally geodesic with respect to the metric $g_{2} = \theta(1)$, we know by \eqref{eqg22} that for some small $\delta>0$, the $\delta$-shift of $g_{2}$ will make the boundary strictly convex. For all $\delta_0>0$ sufficiently small, the $\delta_0$-shift of $g_2$ belongs to the path $\beta$. Given $\delta_0$, there exists a definite $\delta_{1}>0$ (depending on $\delta_0$ but independent of $s$) such that the $\delta_{0}$-shift of $\theta(s)$ has \emph{strictly} convex boundary and positive Ricci curvature for any $s\in[1-\delta_{1},1]$. With this in mind, we construct a new path by deforming $\theta(s)$ through $\delta(s)$-shifts defined by
\begin{align*}
\delta(s)=\left\{
\begin{array}{cc}
0&\text{for}~0\le s<1-2\delta_1,\\
\frac{s-(1-2\delta_1)}{\delta_1}\delta_0&\text{for}~1-2\delta_1\le s<1-\delta_1,\\
\delta_{0}&\text{for}~1-\delta_1\le s\le 1.
\end{array}
\right.
\end{align*}  
Since the metrics $\theta(s)$ and $g_{2}$ share the same normal direction at points in the support of $\varphi(r)$, we see that by choosing $\delta_{0}>0$ small enough, the path of $\delta(s)$-shifts of $\theta(s)$ converges as $s\nearrow 1$ to the $\delta_{0}$-shift of $g_{2}$ which has strictly convex boundary and positive Ricci curvature. Thus, we have constructed a path, denoted by $\sigma$, in $\mathcal{C}$ connecting $\tilde g_2$ to the $\delta_0$-shift of $g_2$.

Analogously, we can shift the path
\begin{align}
\omega(s)={g}_{h}+(1-s)\eta\varphi(r){{g}}^{r}_{h}
\end{align}
between $\tilde{g}_{h}$ and ${g}_{h}$ to a path $\tau$ in $\mathcal{C}$ between $\tilde{g}_{h}$ and the $\delta_0$-shift (choosing $\delta_0$ smaller if necessary) of ${g}_{h}$ which will be a \emph{strictly} convex round spherical cap.

Therefore, we have the following paths of metrics, all belonging to $\mathcal{C}$:
\begin{align}\label{eq:completepath}
g \xrightarrow{\alpha} g_1 \xrightarrow{\beta} \delta_0\text{-shift of } g_2 \xrightarrow{\text{reversing } \sigma} \tilde{g}_2 \xrightarrow{\tilde \gamma} \tilde{g}_h \xrightarrow{\tau} \delta_0\text{-shift of } g_h,
\end{align}
where $\beta$ is the restriction of $\beta$ from $g_1$ to $g_2$ to the path terminating at the $\delta_0$-shift of $g_2$. Thus, $\mathcal{C}$ is path-connected.

 
 \subsection{The moduli space $\mathcal{C}/ \text{Diff}_+$ is contractible}\label{contra} By \cite{Fis70} and \cite{Mil59}, it is enough to show that all homotopy groups of $\mathcal{C}/ \text{Diff}_+$ vanish. We already proved that $\pi_0(\mathcal{C}/ \text{Diff}_+)$ is trivial.

Let $k$ be a positive integer, $S^k$ the $k$-sphere, and $h:S^k \rightarrow \mathcal{C}/ \text{Diff}_+$ any smooth map. Since the image $h(S^k)$ represents a compact set of equivalent classes of metrics modulo (orientation-preserving) diffeomorphisms, we may fix small enough gluing parameters $\varepsilon, \rho>0$ (cf. Section \ref{stepA}) which we can use simultaneously for all equivalent classes of metrics in $h(S^k)$. In particular, for any fixed $0<\xi \leq \varepsilon$, we have a well-defined $\xi$-shift for each equivalent class of metrics in  $h(S^k)$ and, moreover, it is not hard to check that these shifts depend continuously on the equivalence class of metrics in $ \mathcal{C}/ \text{Diff}_+$. Therefore, we have a well-defined continuous family of paths $\alpha,\beta$ in $\mathcal{C}_0/ \text{Diff}_+$ such that, as in equation \eqref{eq:pathhemi}, continuously deform each equivalence class of metrics in $h(S^k)$ to an equivalence class of metrics with positive Ricci curvature and totally geodesic boundary. 
 
As before, we double the manifold and run Ricci flow. As well-known to experts, by now standard arguments imply that Ricci flow continuously retracts, modulo diffeomorphisms, any compact set of equivalent classes of metrics with positive Ricci curvature to the equivalent class of the round metrics. This is used to show that the moduli space of metrics with positive Ricci curvature is contractible, if non-empty. Hence, in our setting, we get a well-defined continuous family of paths $\gamma$ in $\mathcal{C}_0/ \text{Diff}_+$ such that, when composed with $\alpha$ and $\beta$ as in equation \eqref{eq:pathhemi}, continuously deform each equivalence class of metrics in $h(S^k)$ to the equivalence class of the round metric on the hemisphere. 

Finally, we note that the family of paths $\gamma \circ \beta \circ \alpha $ constructed above in $\mathcal{C}_0/ \text{Diff}_+$ can be continuously deformed to a family of paths in the moduli space $\mathcal{C}/ \text{Diff}_+$ from $h(S^k)$ to the equivalence class of a strictly convex round spherical cap metric. This can be seen from the fact that the deformation carried out in equation \eqref{eq:completepath} depends continuously on the initial equivalent class of metrics, so that we can fix uniform parameters $\delta_0,\delta_1>0$ for a compact set of equivalent classes of metrics modulo diffeomorphisms. This shows that every smooth map $h:S^k \rightarrow \mathcal{C}/ \text{Diff}_+$ is homotopically trivial and thus finishes the proof of Theorem \ref{thmrc+}.

\section{Proof of Theorem \ref{thmmain}}\label{sect3}
Let $\mathcal{D}$ denote the space of smooth Riemannian metrics on the three-ball ${M}^3$ with non-negative Ricci curvature and strictly convex boundary. The proof of Theorem \ref{thmmain} is quite similar to the proof of Theorem \ref{thmrc+} once we show how to continuously deform any metric $g\in \mathcal{D}$ near $\partial M =D$ keeping the boundary strictly convex and such that the deformed metrics all have positive Ricci curvature near $D$, and non-negative elsewhere. Indeed, we note that the positivity of Ricci curvature was only used in three instances: in the smoothing of the metric $\bar{g}$ near surfaces $\Sigma_{\pm \rho}$; in the convergence of three-dimensional Ricci flows; and in the construction of the path $\tilde\gamma$ (cf. Section \ref{stepC}). Since non-negative Ricci curvature is preserved under Ricci flow \cite{Ham82}, and the smoothing and perturbation occur only near $D$, we see that the same reasoning still holds if now Ricci curvature is strictly positive just near the boundary and non-negative elsewhere. 

Let $g\in \mathcal{D}$. As before, we consider a neighborhood of the boundary of the form $[-2\varepsilon,0]\times D$, where $\varepsilon>0$ is a small number to be fixed, and write $g$ as
$$g = dr^2+g^r.$$
We next let $f^\varepsilon$ be the function defined on $M$ by 
\begin{align*}
 f^\varepsilon & = \left\{
   \begin{array}{rr}
    \exp \left( - \frac{1}{(r+\varepsilon)^2}\right), & \text{on }  (-\varepsilon, 0]\times D,\\
    0, & \text{otherwise}.
   \end{array}
   \right.
\end{align*}
The function $f^\varepsilon$ is clearly smooth. Furthermore, it has the following property.

\begin{lemma}\label{lemmaconvex}
 For a sufficiently small $\varepsilon>0$ the function $f^\varepsilon$ satisfies 
 $$\He_g(f^\varepsilon)\geq0.$$
 Moreover, $\Delta_g f^\varepsilon>0$ on  $(-\varepsilon,0]\times D$.
\end{lemma}

\begin{proof}
It is clear that $\He_g(f^\varepsilon)\equiv 0$ everywhere outside  $(-\varepsilon,0]\times D$. Let $(r,x)$ be local coordinates in $[-\varepsilon, 0] \times D$. In such coordinates, the Hessian of $f$ reduces to
\begin{align*}
\He_g f^\varepsilon & = \partial_{rr} f^\varepsilon dr\otimes dr - \Gamma^0_{ij}\partial_r f^\varepsilon dx^i \otimes dx^j\\
& = \partial_{rr} f^\varepsilon dr\otimes dr - \partial_r f^\varepsilon \text{\rm{II}}(r), 
\end{align*}
where $\text{\rm{II}}(r)$ is the second fundamental form for $\Sigma_{r}$ with respect to metric $g$ and is strictly negative for $\varepsilon$ sufficiently small such that $\Sigma_{\varepsilon}$ is strictly convex. Therefore, choosing $\varepsilon$ smaller if necessary, the derivatives $\partial_r f^\varepsilon$ and $\partial_{rr} f^\varepsilon$ will be positive, and therefore $\He f^\varepsilon \geq 0$. Lastly, we see on  $(-\varepsilon,0]\times D$ there holds\footnote{Since given any real number $A$, there exists $a=a(A)>0$ such that for $x\in(0,a]$, the function $p(x)= \exp\pl-x^{-2}\pr$ satisfies $p''(x) - A p'(x) >0$.}
$$\Delta_g f^\varepsilon = \partial_{rr} f^\varepsilon - H(r) \partial_r f^\varepsilon >0.$$
\end{proof}

From here on, we fix $\varepsilon$ so that Lemma \ref{lemmaconvex} holds true and, to shorten notation, write $f^{\varepsilon}$ as $f$.

We now define a one-parameter family of metrics $g^s$, $s\in\mathbb{R}$, by
\begin{align}\label{eq:gs}
 g^s & = e^{-2sf} g.
\end{align}
Let $\Ric(g^s)$ denote the Ricci curvature tensor of the metric $g^s$. Then, as in \cite[Section 6.5]{MaNuSm13} and by Lemma \ref{lemmaconvex}, we have
\begin{align*}
 \left. \partial_s \right\vert_{s=0} \Ric(g^s) & = \He_g f + \Delta_g f g \geq 0 \quad \text{on } M.
\end{align*}
Moreover, again by Lemma \ref{lemmaconvex}, we have
\begin{align*}
 \left. \partial_s \right\vert_{s=0} \Ric(g^s) & = \He_g f + \Delta_g f g > 0 \quad \text{on } (-\varepsilon,0]\times D.
\end{align*}
So for $s_1$ sufficiently small, $g^s$ will have positive Ricci curvature on $(-\varepsilon,0]\times D$ for all $s\in[0,s_1]$.

We now check the second fundamental form of the metric $g^s$. Choosing $\partial_r$ to be the (outward-pointing) unit normal vector, and using the index $0$ to denote the direction in $\partial_r$, one computes that
\begin{align*}
 {\rm II}(g^s)_{ij} & = \Gamma_{ij}^0(g^s), \\
 \Gamma_{ij}^0(g^s) & = \Gamma_{ij}^0(g) + g_{ij} s \partial_r f.
\end{align*}
Since when $s=0$, ${\rm II}(g)_{ij} = \Gamma_{ij}^0(g)<0$, we see that for $s_2$ sufficiently small, ${\rm II}(g^s)_{ij} < 0 $ for all $s\in[0,s_2]$. 

Therefore, setting $s_0=\min\{s_1,s_2\}$, we obtain the desired deformation of $g$ by considering $g^s$ for $s\in[0,s_0]$. Note that this deformation is invariant under the action of $\text{Diff}_+$ and thus we may show as before that $\mathcal{D} / \text{Diff}_{+}$ and $\mathcal{D}$ are path-connected.

Finally, we argue that all the higher homotopy groups of $\mathcal{D} / \text{Diff}_{+}$ must also vanish. Let $k$ be a positive integer and $h : S^k \to \mathcal{D} /\text{Diff}_{+}$ any smooth map. The image $h(S^k)$ represents a compact set of metrics modulo diffeomorphisms. Because the Ricci curvature and the boundary convexity are geometric invariants under diffeomorphisms, we can deform an equivalent class of metrics with non-negative Ricci curvature by a family of equivalent class of metrics of positive Ricci curvature while preserving the boundary convexity using equation \eqref{eq:gs}. Moreover, we can fix uniform parameters $\varepsilon, s>0$ for a compact set of metrics modulo diffeomorphisms since it is evident that the deformation given in equation \eqref{eq:gs} depends continuously on the equivalent class of metrics. Therefore, we have continuously deformed the compact set $h(S^k)$ in $\mathcal{D} /\text{Diff}_{+}$ to a compact set of equivalent classes of metrics in ${\mathcal{C}} /\text{Diff}_{+}$. Let's call this deformation $\sigma$. As we have proved that $\pi_k(\mathcal{C} /\text{Diff}_{+}) = 0$, the continuous map $\sigma\circ h: S^k \to \mathcal{C} /\text{Diff}_{+}$ is homotopically trivial, and hence every smooth map  $h: S^k \to \mathcal{D} /\text{Diff}_{+}$ is homotopically trivial. Therefore, Theorem \ref{thmmain} is proved.

\section{Free boundary minimal annuli \\ in convex three-manifolds and Theorem \ref{thmmns}}\label{sec:app}

We recall the main theorem of Maximo, Nunes, and Smith \cite{MaNuSm13}:

\begin{theorem}[\cite{MaNuSm13}]\label{thmmns13}
If $(M,g)$ is a smooth, compact, functionally strictly convex Riemannian three-manifold of non-negative Ricci curvature, then there exists a properly embedded annulus $\Sigma\subseteq M$ which is free boundary minimal with respect to $g$.
\end{theorem} 

The notion of convexity used in \cite{MaNuSm13} can be stated as follows: $(M,g)$ is said to be {functionally strictly convex} whenever there exists a smooth function $f:M\rightarrow[0,1]$ which is strictly convex with respect to the metric $g$ and whose restriction to $\partial M$ is constant and equal to $1$ (recall that $f$ is said to be {strictly convex} with respect to a given metric whenever its Hessian is everywhere positive definite). A functionally strictly convex three-manifold is strictly convex in the usual sense and must be diffeomorphic to the three-ball. 

The interest of this concept lies in the fact that the space of metrics with non-negative (or positive) Ricci curvature and functionally strictly convex boundary is path-connected, which is a necessary prerequisite for the degree theoretic techniques used in \cite{MaNuSm13}. Moreover, the degree theoretic argument in \cite{MaNuSm13} works when the space of metrics is open, \emph{e.g.}, metrics with positive Ricci curvature. Thus, Theorem \ref{thmrc+} and the theory developed in \cite{MaNuSm13} allow us to conclude that any metric $g$ on the ball $M$ with \emph{positive} Ricci curvature and strictly convex boundary must contain a properly embedded annulus $\Sigma\subseteq M$ which is free boundary minimal with respect to $g$. The proof of Theorem \ref{thmmns13} can then be achieved if we can show that every metric $g$ with non-negative Ricci curvature and strictly convex boundary and be smoothly approximated by metrics $g_{k}$ with positive Ricci curvature and strictly convex boundary, since the compactness results of Fraser-Li \cite{FrMa14} would guarantee that the sequence of free boundary minimal annuli $\Sigma_{k}$ with respect to $g_{k}$ would converge, possibly after passing to a subsequence, to a free boundary minimal annulus $\Sigma$ with respect to $g$. This can be done by adapting an idea of Aubin \cite{Aub70} and Ehrlich \cite{Ehr76} to the case of manifolds with strictly convex boundary:

\begin{proposition}
Let $g$ be any smooth metric on the three-ball ${M}$ with non-negative Ricci curvature and strictly convex boundary. Then $g$ can be approximated in the $C^\infty$-topology by a sequence of smooth metrics $g_k$ with positive Ricci curvature and strictly convex boundary. 
\end{proposition}

\begin{proof}
Given any $\varepsilon>0$, as argued in Section \ref{sect3}, we can construct a sequence of metrics $g^{s_i}$ converging to $g$ smoothly as $s_i\searrow 0$ where $g^{s_i}$ has non-negative Ricci curvature on $M$, positive Ricci curvature in $(-\varepsilon,0]\times D$, and strictly convex boundary. For each $g^{s_i}$, we construct a sequence of metrics $g^{s_i}_k$ converging to $g^{s_i}$ with $g^{s_i}_k$ having positive Ricci curvature on $M$ and strictly convex boundary. Then a diagonal argument would yield the sequence of metrics $g_{k}$ as claimed in the proposition.

Let $M' = M\setminus (-\varepsilon,0]\times D$. Then $M'$ is a compact subset of $M$. In particular, $\left.g^{s_i}\right\vert_{M'} = g$ and has non-negative Ricci curvature by hypothesis. We claim that there exist $\delta=\delta(g)>0$ and metrics $g^{s_i}_k$ on $M$ with the properties that $\left\vert g^{s_i}_k - g \right\vert_{C^\infty}<\delta 2^{-k}$ and $\left.g^{s_i}_k\right\vert_{M'}$ has positive Ricci curvature. This is achieved by deforming the metric on $M'$ (and near $\partial M'$ but strictly away from $D = \partial M$) using an idea of Aubin \cite{Aub70} and Ehrlich \cite{Ehr76}: Assuming $g$ has non-negative Ricci curvature, at a point where all Ricci curvatures are positive, one can find, by continuity, a small ball centered at that point such that all Ricci curvatures are positive on this small ball. One then deforms the metric locally in such a way that the positive Ricci curvature is spreaded over a slightly larger ball.

The case $k=0$ is essentially proved in \cite[Theorem 5.1]{Ehr76}. We note that the proof there uses \cite[Theorem 4.3]{Ehr76}, which is stated for $B_{g,R}(p)$ with $R\leq 1$. For our purpose, we want $B_{g,R}(p)$ with $R\leq \varepsilon/3$, and it is not hard to see the local deformation result \cite[Theorem 4.3]{Ehr76} holds for such smaller values of $R$. Consequently, when we apply the local deformation argument, we will only deform the metric in $M'\cup(-\varepsilon,-\frac{2}{3}\varepsilon)\times D$, and so the deformation is strictly away from $D$. One checks the proof of \cite[Theorem 5.1]{Ehr76} and realizes that it holds for $2^{-k}\delta$, $k\geq 1$. Indeed, the proof uses Lemma 3.1 and Theorem 3.5 in \cite{Ehr76}, both of which hold for metrics that are $\delta$-close (in $C^\infty$-topology) to $g$, and therefore also hold for metrics that are $2^{-k}\delta$-close to $g$.
\end{proof}

\section{Appendix}

\subsection{Proof of Lemma \ref{lemmarc}}

We work in dimension $n$ with $n\geq 2$. Since $\bar{g}$ agrees with $g$ everywhere, except in the set $[-\rho,\rho]\times D$, we only need to compute its Ricci curvature at the points in $[-\rho,\rho]\times D$. We introduce some notation: 
\begin{itemize}
\item We will use the index $0$ to denote the direction to $\partial_r$. 
\item We will use latin letters $i,j,k,\ldots,$ to denote directions tangential to $D$. 
\item We will use prime notation for derivatives with respect to  ${r}$.
\end{itemize}
With the above notation it is straightforward to check that the only non-zero Christoffel symbols for $\bar g = dr^2 + \bar{g}^r$ are 
\begin{align*}
\Gamma_{ij}^{k}(\bar{g})&=\Gamma^{k}_{ij}(\bar{g}^{r}),\\
\Gamma_{ij}^{0}(\bar{g})&=-\frac{1}{2}\partial_{r} \bar{g}^{r}_{ij}=-\frac{1}{2}\bar{g}^{'}_{ij},\\
\Gamma_{i0}^{k}(\bar{g})&=\frac{1}{2}(\bar{g}^{r})^{kl}\partial_{r} \bar{g}^r_{il}=\frac{1}{2}\bar{g}^{kl} \bar{g}_{il}^{'},
\end{align*} 
and therefore we have the following explicit expression for the components of the sectional curvature
\begin{align}
K(\partial_{i},\partial_{r})&=\text{\rm Rm}(\partial_{i},\partial_{r},\partial_{i},\partial_{r})=\bar{g}_{ik}R^{k}_{i00}\nonumber\\
&=\bar{g}_{ik}\left(\partial_{i}\Gamma_{00}^k-\partial_{r}\Gamma_{i0}^{k}+\Gamma^{k}_{i\alpha}\Gamma^{\alpha}_{00}-\Gamma^k_{0\alpha}\Gamma^{\alpha}_{i0}\right)\nonumber\\
&=-\frac{1}{2}\bar{g}^{''}_{ii} + \frac{1}{4}\bar{g}^{pl}\bar{g}_{ip}^{'} \bar{g}_{il}^{'} \label{eqsec1}.
\end{align}

From \eqref{eqsec1}, and by \eqref{geq1} and \eqref{eqper}, we get that
\begin{align*}
K(\partial_{i},\partial_{r}) & \ge \frac{\Lambda}{2\rho} \left( \bar{g} ^r_\rho\right)_{ii} + \bar{g}^{pl}(x) b_{ip}(x) b_{il}(x) r^2\\
& \ge \frac{c\Lambda}{\rho}
\end{align*}
for some $c>0$ for $(r,x)\in[-\rho,\rho]\times D$ since the coefficients of $\bar{g}^r$ are uniformly bounded in $\rho$ and $\bar{g}^{pl}(x) b_{ip}(x) b_{il}(x) r^2\ge 0$. This implies that 
\begin{align}\label{eqriccir}
\Ric(\partial_{r},\partial_r) \geq (n-1) \frac{c\Lambda}{\rho}.
\end{align}

Next, for a tangential directional $\partial_i$, we have that

\begin{align}
\Ric(\partial_i,\partial_i) &=  K(\partial_i, \partial_r) + \sum_j K(\partial_i,\partial_j) \nonumber \\
&\geq \frac{c\Lambda}{\rho} + \sum_j K(\partial_i,\partial_j). \label{eqriccitan}
\end{align}
By Gauss' equation,
\begin{align*}
K(\partial_{i},\partial_{j})=K_{\bar{g}^r}(\partial_{i},\partial_{j})+\frac{1}{4}\left( \bar{g}_{ij}^{'} \bar{g}_{ij}^{'}-\bar{g}_{ii}^{'} \bar{g}_{jj}^{'}\right),
\end{align*}
so $K(\partial_i,\partial_j)$ must be uniformly bounded in $\rho$. In fact, the coefficients of the metric $\bar{g}^r$ are uniformly bounded in $\rho$, and so are its derivatives in any tangential direction. This yields that $K_{\bar{g}^r}(\partial_{i},\partial_{j})$ is uniformly bounded in $\rho$. Since $g'$ is also uniformly bounded in $\rho$, so will be $K(\partial_i,\partial_j)$. Therefore, by choosing $C>0$ (possibly a different constant than that in \eqref{eqriccitan}, but still independent of $\rho$), we have:
\begin{align}
\Ric(\partial_i,\partial_i) \geq \frac{c\Lambda}{\rho}-C\rho^2.
\label{eqriccitan2}
\end{align}

Therefore, by \eqref{eqriccir} and \eqref{eqriccitan2}, we can choose $\rho$ sufficiently small to obtain a metric $\bar{g}$ with positive Ricci curvature everywhere except on $\Sigma_{\pm\rho}$.

\bibliographystyle{amsalpha}

\bibliography{bib}

\end{document}